\documentclass{amsart}

\usepackage[british,UKenglish]{babel}
\usepackage{graphicx}         
\usepackage{fancyhdr}
\usepackage{rotating}
\usepackage{amsmath}
\usepackage{amssymb}
\usepackage{amsthm}
\usepackage{tikz}
\usepackage{url}
\usepackage{enumerate}
\usepackage{mathtools}
\usepackage{setspace}
\usepackage{color}
\usepackage[normalem]{ulem}
\usetikzlibrary{positioning}

\newtheorem*{notation*}{Notation}

\newtheorem{innercustomthm}{Theorem}
\newenvironment{customthm}[1]
  {\renewcommand\theinnercustomthm{#1}\innercustomthm}
  {\endinnercustomthm}
  
\usepackage{environ}
\newcounter{quote}

\NewEnviron{myquotenumber}{\vspace{3ex}\par
\refstepcounter{quote}
\hfill\parbox{\dimexpr \textwidth--2cm}
{\centering\small\textit{\BODY}}
\hfill\llap{(\thequote)}\vspace{2ex}\par}

\newtheorem{theorem}{Theorem}[section]
\newtheorem{lemma}[theorem]{Lemma}
\newtheorem{fact}[theorem]{Fact}
\newtheorem*{theorem*}{Theorem}
\newtheorem*{maintheorem*}{Main Theorem}
\newtheorem*{lemma*}{Lemma}

\newtheorem{propo}[theorem]{Proposition}
\theoremstyle{definition}

\newtheorem{question}[theorem]{Question}
\newtheorem{definition}[theorem]{Definition}
\newtheorem{remark}[theorem]{Remark}
\newtheorem*{remark*}{Remark}

\newtheorem{claim}{Claim}
\newtheorem*{claim*}{Claim}

\newcommand{\N}{\mathbb{N}}
\newcommand{\M}{\mathfrak{M}}

\newcommand{\wt}{\widetilde}
\newcommand{\Mgp}{$\wt{\M}_c$--group}
\newcommand{\Mcc}{$\wt{\M}_c$--condition}
\newcommand{\ls}{\lesssim}

\newcommand{\qtext}[1]{\quad \text{#1}\quad}

\title[Solvability of the radical in pseudo--finite $\wt{\M}_c$--groups]{Solvability of the radical in pseudo--finite groups with the DCC on centralizers up to finite index}

\author[N Hempel]{Nadja Hempel}
\address{Heinrich Heine University D\"{u}sseldorf}
\email{nadja.valentin@hhu.de}

\author[U. Karhum\"{a}ki]{Ulla Karhum\"{a}ki}
\address{Universit\'{e} Claude Bernard Lyon 1;  Institut Camille Jordan}
\email{karhumaki@math.univ-lyon1.fr}

\thanks{The second author is supported by ANR project MAS (ANR-25-CE40-5294).}

\begin{document}

\begin{abstract} The subgroup generated by all solvable normal subgroups in a pseudo--finite group with the descending chain condition on centralizers up to finite index is solvable. 
Additionally, there is no finitely generated pseudo--finite group whose definable sections satisfy such a chain condition on centralizers.  \end{abstract}

\maketitle

\section{Introduction}

This paper is dedicated to the study of the \emph{radical} of pseudo--finite groups satisfying the descending chain condition on centralizers up to finite index.  By the  \emph{radical} of a group, we mean the subgroup generated by all  normal solvable subgroups. As products of normal solvable subgroups are solvable, the radical of a finite group remains solvable, but for infinite groups this is no longer true: There might be normal solvable subgroups of larger and larger solvability class, and thus the radical becomes a non--solvable subgroup. Thus, it is natural to ask whether there are any conditions on the group for which the radical remains solvable. As solvability of the radical implies its definability, model theoretic tools can be used to study it. Indeed, it is a classical question in model theory whether the radical of a  group satisfying a given first--order tameness condition is solvable. Early results of Belegradek \cite{Belegradek87} show that in a \emph{group of finite Morley rank} the radical is solvable. This result is generalized by Milliet to all groups with a \emph{supersimple} theory \cite{Milliet2013}. Recently, the second author and Wagner showed that the radical of  a \emph{finite--dimensional group with fine and additive dimension} is solvable \cite{Karhumaki-Wagner}. This is an important class of groups in model theory that has recently been defined \cite{Wagner2020} in order to obtain a common framework for several tame groups which admit a well--behaved notion of dimension, such as groups of finite Morley rank and \emph{$o$--minimal} groups. An easy but important observation in \cite{Karhumaki-Wagner} is that in finite--dimensional groups with fine and additive dimension, the centralizers satisfy the descending chain condition (DCC) up to finite index. That is, in such a group any descending chain of centralizers, so that each member is of infinite index in its predecessor, becomes stationary after finitely many steps.

At the center of model theoretic tameness conditions are the so called \emph{stable}--theories (e.g. complex field). They were introduced by Shelah and they always carry an abstract notion of independence generalizing for example linear independence in vector spaces and algebraic independence in fields. In the late 90's, the tools of stability theory were extended to the wider class of \emph{simple} theories. 

Groups with a model theoretically tame theory often satisfy Noetherian properties. For example, in groups of finite Morley rank and in $o$--minimal groups there are no infinite descending chains of definable subgroups. Stable groups satisfy the DCC on uniformly definable subgroups. In particular, they are $\M_c$--groups, i.e.  the DCC holds for centralizers of single elements. In a group with a simple theory any chain of intersections of uniformly definable subgroups, each having infinite index in its predecessor, has finite length. Again, restricting the latter condition to centralizers of elements, we obtain that groups whose theory is simple satisfy the already mentioned DCC on centralizers up to finite index (in fact this remains true in groups definable in the larger class of \emph{rosy} theories \cite{Onshuus2002}). We call such groups, \Mgp s. Motivated by this, the first author studied in detail \emph{definably hereditarily \Mgp}, e.\ g.  groups in which definable sections are \Mgp s \cite{Hempel2020}. In the introduction of \cite{Hempel2020} she explains in more detail the relationships between groups definable in (super)stable/simple/rosy structures, with proper references.  

A group it called \emph{pseudo--finite} if it is infinite and satisfies all first-order sentences of the common theory of finite groups,  or, equivalently, if it is elementarily equivalent to a non--principal ultraproduct of finite groups. We study pseudo--finite \Mgp s. Our work is motivated by questions on pseudo--finite groups whose theory is simple but ultimately the class of groups we study in this article, namely \Mgp, are defined using purely group theoretic notions.

While the solvability of the radical in an arbitrary stable group is an open problem, in \cite{Tent-Macpherson2007}, Macpherson and Tent proved that in a pseudo--finite stable group the radical is solvable. Soon after, Khukhro \cite{Khukhro2009a} generalized their result to pseudo--finite $\M_c$--groups. In \cite[Question 4.10]{Macpherson2018} it is asked whether the radical of a pseudo--finite group whose theory is simple is solvable. We answer this question positively as we obtain that, not only the radical of a pseudo--finite group with a simple theory, but of all pseudo--finite \Mgp s is solvable. In the proof, we use central arguments taken from \cite{Khukhro2009a}, ideas from \cite{Tent-Macpherson2007}  and techniques of almost group theory developed in \cite{Hempel2020}  to bound uniformly the solvability class of the radicals of finite groups whose ultraproduct is an \Mgp. This combined with an  important result by Wilson (Theorem~\ref{th:wilson--radical}), stating that in the class of finite groups the radical is uniformly definable,  allow us to prove our solvability result:

\begin{customthm}{3.3}The radical of any pseudo--finite \Mgp\ is solvable.\end{customthm}

Yet again, another open problem, originating to a question by Sabbagh, asks whether a pseudo--finite finitely generated group exists. In \cite{Karhumaki2021} the second author uses Khkhuro's result in \cite{Khukhro2009a} to show that there is no finitely generated pseudo--finite $\M_c$--group. By applying Theorem~\ref{th:main} and using similar arguments, we show that there is no finitely generated semi--simple pseudo--finite \Mgp. Precisely, we obtain the following.

\begin{customthm}{3.7}Let $G$ be a pseudo--finite \Mgp. If $G/R(G)$ is also an \Mgp\ then $G$ cannot be finitely generated. In particular, there is no finitely generated pseudo--finite group which is either a semi--simple \Mgp \, or a definably hereditarily \Mgp . \end{customthm}

In Section~\ref{sec:backgroun} we fix our notation and give all necessary background results on pseudo--finite groups and on \Mgp s. Theorems~\ref{th:main} and~\ref{th:fingen} are proven in Section~\ref{sec:results}. Then, in Section~\ref{sec:auxiliary}, we prove some auxiliary results which we believe to be of independent interest.

\section{Prerequisites}\label{sec:backgroun}
\subsection{Tildification}Let $G$ be a group, and $H, K \leqslant G $ be subgroups. The group $H$ is said to be \emph{almost contained} in $K$, written $H \ls K$, if $H \cap K$ has finite index in $H$. The subgroups $H,K \leqslant G$ are \emph{commensurable} if both $H \ls K$ and $K \ls H$. This is denoted by $H \sim K$.

\begin{definition}
Let $G$ be a group, and $H, K \leqslant G $.  We define the following:
\begin{enumerate}[(1)]
  
    \item The subgroup $\wt{C}_K(H)=\{k\in K : H  \sim C_H(k) \}$ is the \emph{almost centralizer} of $H$ in $K$. 
    \item The \emph{almost centre} of $G$ is the subgroup $\wt{Z}(G)=\wt{C}_G(G)$.
\end{enumerate}
We say that \emph{commensurability is uniform} in $\wt{C}_K(H)$ if there is some $m\in \mathbb{N}$ so that if $H  \sim  C_H(k)$ then $|H : C_H(k)| < m$.  
\end{definition}

Note that $\wt C_K(H)$ is not necessarily definable.

The following important results are proven by Hempel.

\begin{fact} \label{fact_AlCen}Let $G$ be a group and $H, K \leqslant G $ be definable subgroups. 
\begin{enumerate}[(1)]
    \item \emph{({\cite[Theorem 2.17]{Hempel2020}}, Symmetry)} \label{item_Sym} Then 
    \[
    H \ls \wt C_G(K)\qtext{if and only if}K \ls \wt C_G(H).
    \]
    \item \emph{({\cite[Theorem 2.27]{Hempel2020}})}\label{item_GenVonNeumannThm} Suppose that 
\begin{itemize}
\item  $H$ normalizes $K$;
\item  $H \leq \wt C_G(K)$;
\item $K \leq \wt C_G(H)$, moreover commensurability is uniform in $\wt{C}_K(H)$.
\end{itemize}
Then the group $[K,H]$ is finite.
\end{enumerate}
\end{fact}

In \cite{Hempel2020}, the first author studies groups in which definable sections satisfy the DCC of centralizers each having infinite index in its predecessor, and calls such groups \Mgp s. While this is a perfectly natural set--up, one might also want to consider a more restrictive class of groups where the DCC property may fail in (definable) quotients. With this in mind, we use the following terminology, which does \emph{not} coincide with the terminology in \cite{Hempel2020} (but coincides with the terminology in other recent papers e.g. \cite{Wagner2020,Karhumaki-Wagner}).

\begin{definition} Let $G$ be a group. 
\begin{enumerate}[(1)]
\item We say that $G$ satisfies the \emph{\Mcc} if there is $m< \omega$ such that there are no $(g_i: i \leqslant m)$ in $G$ so that $|C_G(g_j: j < i): C_G(g_j: j \leqslant i)|\geqslant m$ for all $i \leqslant m$. 
\item If $G$ satisfies the \Mcc, we say that $G$ is an \emph{\Mgp}.
\item If any (definable) section of $G$ satisfies the \Mcc, we say that $G$ is a \emph{(definably) hereditarily \Mgp}. 
\end{enumerate}
\end{definition}

Thus we call groups which are said to be \Mgp\ in \cite{Hempel2020}, definably hereditarily \Mgp.

Note that if $G$ is an \Mgp\ then for any $X\subseteq G$ the commensurability is uniform in $\wt{C}_G(X)$ (if necessary, see \cite[Lemma 4.4]{Wagner2020}) and hence $\wt{C}_G(X)$ is a definable subgroup of $G$. Note also that trivially subgroups of an \Mgp\ are again \Mgp s and \emph{definable} subgroups of definably hereditarily \Mgp s are again definably hereditarily \Mgp s. Arbitrary subgroups of definably hereditarily \Mgp s are \Mgp s but not necessary (definably) hereditarily so.

\subsection{Pseudo--finite groups}\label{Sec:pf} We denote by $\mathcal{L}_{gp}$ the language of groups.

\begin{definition}A \emph{pseudo--finite} group is an infinite group which satisfies every first--order sentence of $\mathcal{L}_{gp}$ that is true of all finite groups. \end{definition}

Typical examples of pseudo--finite groups are torsion--free divisible abelian groups, infinite extraspecial $p$--groups of exponent $p >2$ and rank $n$ and (twisted) Chevalley groups over pseudo--finite fields.

\begin{remark}Some authors allow pseudo--finite groups to be finite. We require that a pseudo--finite group is infinite. Therefore, given a pseudo--finite group $G$, any definable subgroup of $G$ is either finite or pseudo--finite. Moreover, given a definable normal subgroup $N$ of $G$, each of the groups $N$ and $G/N$ is either finite or pseudo--finite (see \cite[Lemma 2.16]{Point-Houcine2013}).
\end{remark}

By \L o\'s' Theorem, an infinite group (resp.\ $\mathcal{L}$--structure) is pseudo--finite if and only if it is elementarily equivalent to a non--principal ultraproduct of finite groups ($\mathcal{L}$--structures). Consistently, given any property $P$ of groups, we call a group $G$ \emph{pseudo--$P$} if it is elementary equivalent to an ultraproduct $ \prod_{i\in I} G_i /\mathcal{U}$ where, for $\mathcal{U}$--many $i$, the group $G_i$ satisfy the property $P$.

\subsection{The subgroup $R(G)$}\label{sec:radical}For any group $G$ denote by $R(G)$ the subgroup generated by all solvable normal subgroups of $G$. As mentioned in the introduction, we call this subgroup the \emph{radical} of $G$. In literature terminology varies; some authors call the radical of a group `the solvable radical'. Trivially, for any finite group the radical is solvable but $R(G)$ might be non--solvable for an infinite group $G$. For finite groups, we moreover have the following important result:

\begin{theorem}[Wilson \cite{Wilson2009}]\label{th:wilson--radical}There exists a parameter--free first--order formula $\varphi_{R}(x)$ so that given any finite group $H$ we have $\varphi_{R}(H)=R(H)$. \end{theorem}

Let $G\equiv \prod_{i\in I} G_i /\mathcal{U}$ be a pseudo--finite group and consider $\varphi_{R}(G)$. Clearly the pseudo--(finite and solvable) group $\varphi_{R}(G)$ is a definable normal subgroup of $G$ which contains all definable solvable normal subgroups of $G$. In general, the subgroup $\varphi_{R}(G)$ need not to be solvable as there may not be a common bound on the derived lengths of $\mathcal{U}$--many of the finite solvable groups $\varphi_{R}(G_i)$. Indeed, this can happen even in the case of groups whose first--order theory is rather tame, as demonstrated in \cite{Tent-Macpherson2012} where the authors give an example of a pseudo-finite group $G$ whose first--order theory is NIP and the subgroup $\varphi_{R}(G)$ is non--solvable. Note also that $\varphi_{R}(G)$ is trivial if and only if $G$ is \emph{semi--simple}, that is, has no non--trivial abelian normal subgroup (see \cite[Lemma 2.17]{Point-Houcine2013}). Since semi--simplicity is first--order expressible by the sentence $(\forall x\neq 1)(\exists y)([x, x^y]\neq 1)$, it is easy to see that $R(G) \leqslant \varphi_R(G)$ and that $R(G)$ is solvable if and only if $\varphi_R(G)$ is solvable, and when $R(G)$ is solvable we have $R(G) =\varphi_R(G)$. Therefore, to show that the radical of a pseudo--finite group is solvable, it is enough show that $\varphi_{R}(G)$ is solvable.

\section{Solvability of the radical}\label{sec:results}Let $G$ be a non--abelian group such that for any $g\in G\setminus Z(G)$ the centralizer $C_G(g)$ is solvable. We say that such a group $G$ is \emph{$c$--bounded} (short for centralizer--bounded) if there is $\ell\in \N$ so that for any $g\in G\setminus Z(G)$ the solvable centralizer $C_G(g)$ is of derived length at most $\ell$. In this case the integer $\ell$ is called the \emph{$c$--derived length} of $G$.

\subsection{Finite solvable $c$--bounded groups} 
In what follows, we extract some parts of the proof of \cite[Theorem 1]{Khukhro2009a}. Most of the arguments are taken one--to--one. However we do offer the full proof for completeness. As a result we show that the derived length of a finite solvable $c$--bounded group depends only on its $c$--derived length. This is then applied in the proof of Theorem~\ref{th:main}. Further, we hope that such a statement can be useful in the future e.g. in the studies of locally finite \Mgp s (cf. \cite{Borovik-Karhumaki2019}).

We now fix the following notation for any finite solvable group $H$:
\begin{itemize}
    \item $h(H)$ denotes the Fitting height of $H$. 
    \item $F_d(H)$ is the $d$--th member in the upper Fitting series of $H$.
\end{itemize}

\begin{theorem}[Thompson's Theorem, \cite{Thompson64}]\label{th:Thosmpson} Suppose that a finite solvable group $G$ has an automorphism $\alpha$ of prime order coprime to the order of the group. Then $h(G) \leq 5 h(C_G(\alpha))$.
\end{theorem}

\begin{propo}\label{pr:Khukhro}Let $H$ be a finite solvable $c$--bounded group of $c$--derived length $\ell$. Then the derived length of $H$ is at most $6\ell+3$. \end{propo}

\begin{proof}Note first that if $N$ is a nilpotent subgroup of $H$ then the derived length of $N$ is at most $\ell +1$: We may assume that $N$ is not abelian. Then we can find $g \in Z_2(N)\setminus N$. By Witt's identity, we have that $C_N(g)$ contains $[N,N]$. As $H$ is $c$--bounded group of $c$--derived length $\ell$,  $C_N(g)$  is solvable of class at most $\ell$. Thus, $N$ is at most $(\ell+1)$--solvable.

By the above the derived length of any Sylow $p$--subgroup of $H$ is at most $\ell+1$. By Hall--Higman Theorems (the full statement can be found in \cite{HallHigman}) we have that \begin{itemize}
\item the $p$--length of $H$ is at most $\ell+1$ for each prime $p$. 
\end{itemize}Since it is clear that the derived length of $H$ is at most $h(H)$ plus the maximum $p$--length of $H$, we also have that \begin{itemize}
\item the derived length of $H$ is at most $(\ell +1)+h(H)$.
\end{itemize} So it suffices to show that $h(H)$ is at most $5\ell +2$. 

Set $h=h(H)$. We may assume that $h>1$. We now fix a prime divisor $p$ of $|H/F_{h-1}(H)|$, the last non--trivial Fitting quotient of $H$. Let $P$ be the Hall $p'$--subgroup of $F_{h-1}(H)$. It is well-known that $h(H) \leqslant h(P)+2$: Of course $h=h(F_{h-1}(H))+1$. At the same time $F_{h-1}(H)=O_p(F_{h-1}(H))P$, where $O_p(F_{h-1}(H))$ is the largest normal $p$-subgroup of $F_{h-1}(H)$ hence nilpotent and of Fitting height $1$. This means that $h(F_{h-1}(H))$ is at most $1+ h(P)$. Hence $h$ is at most $h(P)+2$. So it remains to bound $h(P)$ by $5 \ell$.

As the Hall $p'$--subgroups of $F_{h-1}(H)$ are conjugate in $F_{h-1}(H)$ we have that $H = N_{H}(P)F_{h-1}(H)$ by an analogue of Frattini's argument. Therefore we can choose a $p$--element $k$ in $N_{H}(P)\setminus F_{h-1}(H)$. If $P \leq C_{H}(k)$, we are finished, because $k$ is non--central in $H$ and thus $C_{H}(k)$ is $\ell$--solvable. Thus, we may assume, that $|N_{H}(P)/C_{H}(P)|$ is divisible by $p$. Let $\gamma$ be the automorphism of order $p$ in $\operatorname{Aut}(P)$ induced by conjugation by some element $g \in N_{H}(P)$. Then $C_{P}(\gamma)= C_{P}(g) $ has derived length at most $\ell$. In particular the Fitting height of $C_{P}(\gamma)$ is bounded by $\ell$. Now, by Theorem~\ref{th:Thosmpson}, the Fitting height of $P$ is at most $5h(C_{P}(\gamma)) \leqslant 5 \ell$.\end{proof}

\subsection{The main results} We now prove our main theorem:

\begin{theorem}\label{th:main}
The radical of any pseudo--finite \Mgp\ is solvable.
\end{theorem}

\begin{proof}Let $G\equiv \prod_{i\in I}G_i/\mathcal{U}$ be a pseudo--finite \Mgp. As is explained in Section~\ref{sec:radical}, we just need to show that the pseudo--(finite and solvable) subgroup $\varphi_R(G)$ is solvable, so we may set $G=\varphi_R(G)$. Towards a contradiction, suppose that $G$ is non--solvable. Since $G$ is solvable precisely when $\prod_{i\in I}G_i/\mathcal{U}$ is solvable, we may assume that $G=\prod_{i\in I}G_i/\mathcal{U}$. Using the \Mcc, we may replace $G$ by the centralizer of finitely many elements and assume that any centralizer $C_G(g)$ of an element $g\in G$ of infinite index is solvable whereas $G$ is still non--solvable. Moreover by saturation of ultraproducts and compactness, there is $t \in \N$ such that any centralizer of infinite index is solvable of derived length at most $t$. 

\begin{claim}\label{claim:MainTheorem} We may assume that $Z(G) = \wt Z(G)$. 
\end{claim}

\begin{proof}
In the proof of the claim, we will replace $G$ twice by a definable subgroup of finite index. Note, that in this case we keep the following assumptions:

\begin{itemize}
    \item $G$ is non--solvable;
    \item $G =\varphi_R(G)= \prod_{i\in I} \varphi_R(G_i)/\mathcal{U} $. In particular $G$ is pseudo--solvable;
    \item Any centralizer of infinite index is solvable of derived length at most $t$.
\end{itemize}

Trivially, one has that $\wt Z (G) \leq   \wt C_G (G)$. Using symmetry of the almost centralizer (Fact \ref{fact_AlCen}(\ref{item_Sym})), we obtain that $ G \ls \wt C_G (\wt Z(G)) $. Thus replacing $G$ by the finite index subgroup $\wt C_G (\wt Z(G)) $, we may assume that
\[
\wt Z (G) \leq  \wt C_G (G) \text{ and } G \leq \wt C_G (\wt Z(G)).
\]
Thus $[ G, \wt Z(G)]$ is finite by Fact \ref{fact_AlCen}(\ref{item_GenVonNeumannThm}) and it is contained in the normal subgroup $\wt Z(G)$ of $G$. We may again replace $G$ by the finite index subgroup $C_G([ G, \wt Z(G)])$. Thus $[ G, \wt Z(G)] \leq Z(G)$ and $ \wt Z(G)$ is (finite and central)--by--central. In particular $G = C_G(\wt Z(G)/Z(G))$. As trivially $G = C_G(Z(G))$, the quotient $G/C_G(\wt{Z}(G))$ consists of automorphisms of $\wt Z(G)$ which fixes $Z(G)$ and $\wt{Z}(G)/Z(G)$ pointwise. Thus $G/C_G(\wt{Z}(G))$ embeds in a direct power of the abelian group $Z(G)$ and it therefore abelian itself, in particular solvable. As $G$ is non--solvable, the centralizer $C_G(\wt Z(G))$ must be non--solvable. It is still pseudo--solvable and any centralizer of infinite index is solvable of derived length at most $t$. So we may continue to work with $H:=C_G(\wt Z(G))$, i.e. $\wt Z(G) \subseteq Z(H)$, and show that 
    \[
    Z(H) = \wt Z(H).
    \]
The inclusion from left to right is obvious. So let $h \in  \wt Z(H)$. Then $[H: C_H(h)] < \infty$. As $H$ is pseudo--solvable and non--solvable, we have that $C_H(h)$ is non--solvable and hence $C_G(h)$ is non--solvable as well. Thus $[G: C_G(h)]$ must be finite, i.e. $h \in \wt Z(G) \subseteq Z(H)$. So, replacing $G$ by $H$ we may assume that the claim holds.\end{proof}
By the above, for any non--central $g\in G \setminus Z(G)$ (which exists, as $G$ is assumed to be non--solvable) the centralizer $C_G(g)$ is of infinite index in $G$ and hence is solvable of class at most $t$. That is, $\mathcal{U}$--many of the finite solvable non--abelian groups $G_i$ is $c$--bounded of $c$--derived length $t$. Proposition~\ref{pr:Khukhro} then implies that $\mathcal{U}$--many of the finite solvable groups $G_i$ is of derived length at most $6t+3$. This means that $G$ is solvable  -- a contradiction.
\end{proof}

\subsection{A non--existence result for finitely generated pseudo--finite groups} There are no known examples of finitely generated pseudo--finite groups. Indeed, it is a question by Sabbagh whether one exists and there are several non--existence results related to this question. In particular, it is known that neither solvable nor simple finitely generated pseudo--finite groups exist \cite[Proposition 3.14]{Point-Houcine2013}. In \cite{Karhumaki2021} the second author shows that there is no finitely generated pseudo--finite $\M_c$--group. Below, using similar arguments as in \cite{Karhumaki2021}, we isolate conditions which ensure that a pseudo--finite group is not finitely generated. Then, using these conditions and Theorem~\ref{th:main}, we show that a finitely generated pseudo--finite definably hereditarily \Mgp\ does not exist. This in particular shows that there is no finitely generated pseudo--finite group whose first--order theory is simple or merely rosy. 

We will need the following deep result on finite simple groups.

\begin{theorem}[Ellers and Gordeev {\cite{Ellers-Gordeev1998}} and Hs\"{u} (Xu) \cite{Cheng-hao98}]\label{fact:Ellers}Let $H$ be a finite simple group. Suppose that $H$ is either of (twisted) Lie type over a field with more than $8$ elements, or ${\rm Alt}_n$ for $n\geqslant 5$. Then there exists a conjugacy class $C$ such that $H=CC$. \end{theorem} 

A group is called \emph{definably simple} if it has no proper definable non--trivial normal subgroup. Obviously any pseudo--finite group $G\equiv \prod_{i\in I} G_i/\mathcal{U}$ where $\mathcal{U}$--many of the finite groups $G_i$ are simple is definably simple.

\begin{fact}[Palac\'{i}n {\cite[Lemma 2.2]{Palacin2018}}]\label{fact:Palacin} Let $G$ be a finitely generated pseudo--finite group and assume that it contains an infinite definable definably simple non--abelian normal subgroup $N$. Then the group $N$ is a finitely generated pseudo--finite group.
\end{fact} 

The \emph{socle} of a \emph{finite} group $G$, denoted by ${\rm Soc}(G)$, is the subgroup generated by all minimal normal non--trivial subgroups of $G$. If $G$ is semi--simple then ${\rm Soc}(G)$ is a finite direct product $M_{1}\times \cdots \times M_{\ell}$ of minimal normal subgroups, where for each $m\in \{1, \ldots, \ell\}$ the group $M_{m}$ is a finite direct product of isomorphic non--abelian finite simple groups. That is, ${\rm Soc}(G)$ is a finite direct product of non--abelian simple groups.

\begin{propo}\label{propo:conditions}Let $G\equiv \prod_{i\in I} G_i/\mathcal{U}$ be a pseudo--finite group satisfying the following.\begin{enumerate}[(a)]
\item $\varphi_R(G)$ is solvable.
\item There is $\ell\in \mathbb{N}$ so that for $\mathcal{U}$--many $i$ the socle $Soc(G_i/\varphi_R(G_i))$ is a product of at most $\ell$ many non-abelian simple groups.
\item Set $\overline{G}=G/\varphi_R(G)$. If $\overline{H}\equiv \prod_{i\in I} \overline{H}_i/\mathcal{U}$ is a definable subgroup of $\overline{G}$ so that for $\mathcal{U}$--many $i$ the finite group $\overline{H}_i$ is a simple group of Lie type, then $\overline{H}$ is simple.\end{enumerate}
Then $G$ is not finitely--generated.
\end{propo}

\begin{proof}Towards a contradiction, suppose that $G\equiv \prod_{i\in I} G_i/\mathcal{U}$ is finitely generated. We may assume that $G/\varphi_R(G)$ is infinite: Otherwise, by (a), $\varphi_R(G)$ is a solvable subgroup of finite index, whence a finitely generated pseudo--finite solvable group. Thus it is finite by \cite[Proposition 3.14]{Point-Houcine2013}. This would imply that $G$ is finite -- a contradiction. 

Denote $\overline{G}=G/\varphi_R(G)\equiv \prod_{i\in I} (G_i/\varphi_R(G_i))/ \mathcal{U}$ and $\overline{G}_i=G_i/\varphi_R(G_i)$. Now $\overline{G}$ is a finitely generated semi--simple pseudo--finite group. The socle ${\rm Soc}(\overline{G}_i)$ of the semi--simple finite group $\overline{G}_i$ is a direct product of $j_i$ many minimal normal subgroups $((M_{i})_{r})_{ r =1 , \dots, j_i}$ each of which is a product of $\ell_{i,r}$ many isomorphic simple non--abelian groups. Thus, for $ r =1 , \dots, j_i$ there is a simple non--abelian group $(S_i)_r$ such that 
    \[
    {\rm Soc}(\overline{G}_i) \cong \bigtimes_{l_{i,1}}(S_i)_1 \times \dots \times  \bigtimes_{l_{i,j_i}}(S_i)_{j_i}
    \]
By (b), $\sum_{r=1}^{j_i}\ell_{i,r}\leq \ell$ modulo $\mathcal{U}$. By semi--simplicity, $C_{\overline{G}_i}({\rm Soc}(\overline{G}_i))=1$, so $\overline{G}_i \hookrightarrow {\rm Aut}({\rm Soc}(\overline{G}_i))$. Note that the size of $\overline{ G_i}$ grows without a bound when $i$ varies: If it were bounded, then $\overline{G}$ was finite, contradicting our assumption. By the embedding above, we obtain that $|{\rm Soc}(\overline{G}_i)|$ grows without a bound when $i$ varies. Furthermore, as  for each $i$ the number of simple non-abelian components in the direct product of ${\rm Soc}(\overline{G}_i)$ is bounded by $\ell$, there is $1\leq r_i \leq j_i$, such that $|(S_i)_{r_i}|$ grows without a bound when $i$ varies. Additionally, for $\mathcal U$--many $i$, there is $m \in \N$ such that $\ell_{i,r_i} = m$. Therefore, we may assume that $(S_i)_{r_i}$ is normal in $\overline{G}_i$: Since $\overline{G}_i$ permutes the  $m$ many components of $(M_i)_{r_i}$, there is $n\in \N$ depending only on $m$ so that after replacing $\overline{G}_i$ with an index $n$ subgroup we have that each component of $(M_i)_{r_i}$ is normal in $\overline{G}_i$. That is, after replacing $\overline{G}$ with a finite index subgroup, $\overline{G}$ is still finitely generated and we also have that each $(S_{i})_{r_i}$ is normal in $\overline{G}_i$.

By the Classification of Finite Simple Groups, modulo $\mathcal{U}$, each finite simple group $(S_{i})_{r_i}$ is either ${\rm Alt}_n$ with $n \geqslant 5$ or a group of (twisted) Lie type over a field with more than $8$ elements. So, by Theorem~\ref{fact:Ellers} there is $x_{i} \in (S_{i})_{r_i}$ such that $$(S_{i})_{r_i}=x_{i}^{(S_{i})_{r_i}}x_{i}^{(S_{i})_{r_i}}.$$ As $(S_{i})_{r_i} \unlhd \overline{G}_i$, we have $$x_{i}^{\overline{G}_{i}}x_{i}^{\overline{G}_{i}}  = (S_{i})_{r_i}.$$ This means that $(S_{i})_{r_i}$ is uniformly definable in $\overline{G}_i$. So there is a definable definably simple non--abelian subgroup $\overline{S}\equiv \prod_{i\in I}(S_{i})_{r_i}/\mathcal{U}$ of the finitely generated \Mgp\ $\overline{G}$. By Fact~\ref{fact:Palacin} the group $\overline{S}$ is finitely generated. So $\overline{S}$ cannot be simple as there are no finitely generated simple pseudo--finite groups by \cite[Proposition 3.14]{Point-Houcine2013}. Thus (c) implies that $\mathcal{U}$--many of the finite groups $(S_{i})_{r_i}$ are alternating groups. However a result by Kh\'{e}lif  \cite{Khelif} yields that there is no infinite finitely generated group which is elementarily equivalent to an ultraproduct of finite alternating groups -- a contradiction.\end{proof}

\begin{theorem}\label{th:fingen}Let $G$ be a pseudo--finite \Mgp. If $G/R(G)$ is also an \Mgp\ then $G$ cannot be finitely generated. In particular, there is no finitely generated pseudo--finite group which is either a semi--simple \Mgp \, or a definably hereditarily \Mgp .
\end{theorem}
\begin{proof} It is enough to show any pseudo--finite group $G$ which is \Mgp \, whose definable section $G/R(G)$ is also an \Mgp \, satisfy (a)--(c) from Proposition~\ref{propo:conditions}. The in particular part follows trivially. 

Theorem~\ref{th:main} implies (a). Since the semi--simple group $G/R(G)$ is an \Mgp \, (b) also holds (if necessary, see \cite[Lemma 4.1]{Zou2020}). Suppose that $\overline{H}\equiv \prod_{i\in I} \overline{H}_i/\mathcal{U}$ is a definable subgroup of $\overline{G}$ so that for $\mathcal{U}$--many $i$ the finite group $\overline{H}_i$ is a simple group of Lie type. Since there are only finitely many Lie types, modulo $\mathcal{U}$, each of the simple group $\overline{H}_i$ is of the same Lie type. If we also have that modulo $\mathcal{U}$ each of the simple group $\overline{H}_i$ is of the same Lie rank, then $\overline{H}$ is a simple group by \cite{Wilson1995}. So in this case (c) also holds. Suppose contrary that there is no common bound on the Lie ranks of the $\mathcal{U}$--many of the groups $\overline{H}_i$. Then there is no common bound on the derived lengths of $\mathcal{U}$--many of the Borel subgroups $\overline{B}_i$ of $\overline{H}_i$ (this is a well--known fact which easily follows from the structure of the Borel subgroup, see e.g. \cite[Section 8]{Carter1971}). So $\prod_{i\in I} \overline{B}_i/\mathcal{U}$ is a pseudo--(finite and solvable) \Mgp \,which is non--solvable. But Theorem~\ref{th:main} implies that $\prod_{i\in I} \overline{B}_i/\mathcal{U}$ is solvable -- a contradiction. \end{proof}

Finally let us note that, while one cannot expect to prove that arbitrary quotients of an \Mgp\ $G$ satisfy the \Mcc, one may ask if this holds for the specific quotient $G/R(G)$. In the case of pseudo--finite $\M_c$--groups the quotient by the radical remains to be an $\M_c$--group \cite[Proposition 2.1]{Buturlakin2019} (see also \cite{Borovik-Karhumaki2019}). If the same holds for the \Mgp  s, then Theorem~\ref{th:fingen} immediately generalizes to \Mgp  s.

\begin{question} Is the quotient $G/R(G)$ of an \Mgp\ $G$ again an \Mgp ? Is this true for a pseudo--finite $G$? In particular, does a finitely generated pseudo--finite \Mgp\ exist?\end{question}

\section{Auxiliary results}\label{sec:auxiliary} In this section we show some related results which we believe to be of independent interest and/or can hopefully be used to prove general results on (pseudo--finite) definably hereditarily \Mgp s.

\subsection{Unbounded definably hereditarily \Mgp s}\label{sec:moregeneral} We now borrow terminology and techniques from \cite{Tent-Macpherson2007}. A group $G$ is called \emph{unbounded} if for each $d\in \mathbb{N}$ it has a solvable normal subgroup of derived length greater than $d$; otherwise $G$ is \emph{bounded}. Note that the radical $R(G)$ of a group $G$ is solvable if and only if $G$ is bounded.

As far as the authors are aware, there are no examples of unbounded \Mgp s. In model theory, it is a well--known open problem whether the radical of any stable group is solvable. As stable groups are $\M_c$--groups this means that it is not known whether the radical of a definably hereditarily $\M_c$--group $G$ is solvable. Given this we of course do not know whether the radical of a definably hereditarily \Mgp\ is solvable. We do \emph{not} claim that we have a strategy towards solving this problem. Yet, in what follows, we show a result, which has the potential of being one of the initial steps towards solving this problem. Namely, we show that if there is an unbounded definably hereditarily \Mgp\ then there is one with a definable and (almost) faithful action on a definable abelian group (Proposition~\ref{lem_GroupActionAH}).

We start with some easy lemmas.

\begin{lemma}\label{lem_AlCenProps} Properties of the almost centralizer:
\begin{enumerate}[(1)]
    \item \label{item_AlCenProps_ModuloFinite} Let $G$ be a group and $A, B, F$ be such that $F \trianglelefteq B\trianglelefteq A$, $F \trianglelefteq G$, and $F$ as well as $A/B$ is finite. Then $\wt C_G(A) =\wt  C_G (B/F)$.
    \item \label{item_AlCenProps_FiniteCom} Let $G$ be an \Mgp\ and $H$ be a definable subgroup of $G$. Then there exists a definable subgroup $H_0$ of $H$ of finite index such that $[H_0, \wt C_G(H)] $ is finite. Also, if $H$ is a normal subgroup of $G$ then so is $H_0$.
\end{enumerate}
\end{lemma}

\begin{proof} Let $g \in \wt C_G (B/F)$, then $[B,g]/F$ is finite. As $F$ is finite, so is $ [B,g]$. Thus $C_B(g)$ has finite index in $B$. As $B$ has finite index in $A$, so does $C_B(g)$ and in particular $C_A(g)$. Thus $g \in \wt C_G(A)$. As the other inclusion is obvious, this proves (1).

For (2), note first that we trivially have $\wt C_{G} (H) \leq \wt  C_{G} (H)$. By Fact \ref{fact_AlCen}(\ref{item_Sym}), we obtain $H \ls \wt C_{G} (\wt C_{G} (H))$. Let $  H_0: = \wt C_{H} (\wt C_{G} (H)) $, which is definable as $G$ is an \Mgp. Clearly, if $H$ is a normal subgroup of $G$ then so is $H_0$. Moreover, as $H_0 \sim H$ and by the choice of $H_0$, we have $H_0 \leq \wt  C_G (\wt C_G (H))$ and $\wt C_{G} (H) \leq \wt  C_{G} (H_0)$. Moreover $H_0$ normalizes $\wt C_{G} (H)$ and commensurability is uniform in $\wt  C_{\wt C_{G} (H)} (H_0)$ by definability of the almost centralizer. Hence Fact \ref{fact_AlCen}(\ref{item_GenVonNeumannThm}) yields that $ [\wt  C_{G} (H), H_0]$ is finite.\end{proof}

\begin{lemma}\label{lemma:Tent--Mcpher}
Let $G$ be a group and $1= R_0 < R_1 < \dots < R_m =R $ be a chain of normal subgroups of $G$ such that each $R_{l+1}/R_{l}$ is either abelian or finite, and $G/ C_G(R)$ is unbounded.  
Then there exists $l < m$ such that $$C_G(R_l) \Big / C_{C_G(R_l)}(R_{l+1}/R_l)$$ is unbounded and $R_{l+1}/R_{l}$ is infinite.
\end{lemma}

\begin{proof}We argue as in \cite[Proof of Lemma 4.2]{Tent-Macpherson2007}: Let $l$ be maximal so that $G/C_G(R_l)$ is bounded. Then $C_G(R_l)\Big/ C_G(R_{l+1})$ is unbounded and $$ C_G(R_{l+1}) \leq C_{C_G(R_l)} (R_{l+1}/R_l) \leq C_G(R_l).$$ Since $  C_{C_G(R_l)} (R_{l+1}/R_l)/  C_G(R_{l+1})$ consists of automorphisms of $R_{l+1}$ which fix $R_l$ and $R_{l+1}/R_l$ pointwise, it embeds in a direct power of the bounded group $R_l$, so it is bounded. It follows that the group $$C:=C_G(R_l) \Big / C_{C_G(R_l)}(R_{l+1}/R_l) $$ is unbounded. Note also that $C$ acts faithfully on $R_{l+1}/R_l$.  Therefore, if $R_{l+1}/R_l$ was finite then $C$ would be finite too, which is not the case as it is unbounded. 
\end{proof}

\begin{propo}\label{lem_GroupActionAH}
Let $G$ be an unbounded definably hereditarily \Mgp. Then there are definable sections $A$ and $H$ of $G$ and a definable action of $H$ on $A$ such that
\begin{itemize}
    \item $A$ is abelian;
    \item the action of $H$ on $A$ is faithful and by conjugation;
    \item $H$ is an unbounded  definably hereditarily  \Mgp;
    \item $\wt C_G(A) = C_G(A)$.
\end{itemize}
\end{propo}

\begin{proof}Since $G$ is unbounded there is an infinite chain of solvable normal subgroups $T_i \trianglelefteq G$ of derived length $n_i$ with $n_i \overset{i \to \infty}{\longrightarrow} \infty$. By \cite[Theorem 3.2(3)]{Hempel2020} there are definable solvable normal subgroups $R_i \trianglelefteq G$ of derived length $m_i$ with $n_i \leq m_i \leq 3n_i$ (thus $m_i  \overset{i \to \infty}{\longrightarrow} \infty$) such that, after taking intersections if necessary, $R_i \leq R_{i+1}$. Set $R = \bigcup R_i$ and note that $R \trianglelefteq G$. By properties of almost centralizers (if necessary, see \cite[Propoerties 2.11]{Hempel2020}) we have that
    \[
     \wt C_G(R) = \bigcap \wt C_G(R_i),
    \]
and $\wt C_G(R_i) \geq \wt C_G(R_{i+1})$. Hence, as the almost centralizer is definable in any \Mgp , we obtain by compactness that $ \wt C_G(R) = \wt C_G(R_n)$ for some $n \in \N$. Since the commutator subgroup $\wt Z(R_n)'$ is finite (Fact \ref{fact_AlCen}(\ref{item_GenVonNeumannThm})), $\wt Z(R_n)$ is finite--by--abelian. Thus $R_n/\wt Z(R_n) $ is non--solvable and hence $G/ \wt C_G(R_n)$ is unbounded. Set $F:=\wt Z(R_n)'$. Now we work in $\overline G = G/F$, which is still a definably hereditarily \Mgp\, and consider $\overline R_n := R_n/F \leq \overline G$. We have $\wt Z (\overline R_n) = \wt Z(R_n) /F$: For any $r \in R_n$,  \[
    [R_n : C_{R_n}(r)] \leq [\overline R_n : C_{\overline R_n} (rF)] \cdot |F| \qtext{ and } [\overline R_n : C_{\overline R_n} (rF)] \leq [R_n : C_{R_n}(r)].
    \] Thus $[R_n : C_{R_n}(r)]$ is finite if and only if $[\overline R_n : C_{\overline R_n} (rF)]$ is finite. Hence $rF \in  \wt Z(\overline R_n)$ if and only if $ r \in \wt Z(R_n)$. So $\wt Z (\overline R_n) = \wt Z(R_n) /F$ and hence $\wt{C}_G(R_n)=C_G(\overline{R}_n)$. So, if we replace $G$ by $\overline{G}$ and $R_n$ by $\overline{R}_n$, we may assume that $G/C_G(R_n)$ is unbounded.

Now, by \cite[Proposition 3.8]{Hempel2020} the solvable definably hereditarily \Mgp\ $R_n$ contains a chain of definable normal (in $G$) subgroups $$1=S_0 \leq S_1 \leq \cdots S_{m_n}=R_n,$$ where $S_{l+1}/S_l$ is finite--by--abelian for each $l\in \{0, \ldots m_i\}$.

Let $l  \in \{ 0, \dots , m_{n}-1\}$. In what follows, we show that there are definable normal (in $G$) subgroups $ S_{l}\leq S_{l}^+\leq S_{l+1}^- \leq S_{l+1}$ such that $ S_{l+1}/S_{l+1}^-$ and $S_{l}^+/S_l$ are finite, $S_{l+1}^-/S_{l}^+$ is abelian, and $\wt C_{G} (S_{l+1}^-/S_{l}^+)= C_{G} ( S_{l+1}^-/S_{l}^+)$. To ease notation, we work modulo $S_l$ and thus may assume that $S_l= 1$ and $S_{l+1}$ is finite-by-abelian. In particular $S_{l+1} 
\leq \wt  C_{G} (S_{l+1})$.

By Lemma \ref{lem_AlCenProps}(\ref{item_AlCenProps_FiniteCom}) there is a definable normal (in $G$) subgroup $S_{l+1}^-$ of $S_{l+1}$ of finite index such that $S_l^+ := [\wt  C_{G} (S_{l+1}), S_{l+1}^-]$ is finite (thus definable). Clearly $S_l^+$ is also normal in $G$. Set $A_{l+1}:= S_{l+1}^-/ S_l^+$, which is clearly abelian. By Lemma \ref{lem_AlCenProps}(\ref{item_AlCenProps_ModuloFinite}) we have that $\wt C_{G} (A_{l+1})= \wt C_{G} (S_{l+1})$. Moreover, it is easy to see that $\wt C_{G} (A_{l+1}) =  C_{G} (A_{l+1})$: The inclusion from right to left is obvious. If $g \in  \wt C_{G} (A_{l+1})= \wt C_{G} (S_{l+1})$ then $[g, S_{l+1}^-] \subset  [\wt  C_{G} (S_{l+1}), S_{l+1}^-]= S_l^+$. Thus $[g, S_{l+1}^-]/S_l^+ = \bar 1 $ and  $g \in  C_{G} (A_{l+1})$. 

 We obtained a new descending series
    \[ S_0 \leq S_0^+ \leq S_1^- \leq S_1 \leq S_1^+ \leq S_2^- \leq S_2 \leq \dots \leq S_{m_n -1} \leq S_{m_n -1}^+ \leq S_{m_n}^- \leq S_{m_n},
    \]
where  $S_0=1$, $S_{m_n} = R_n$, $ S_{l+1}/S_{l+1}^-$ and $S_{l}^+/S_l$ are finite, $S_{l+1}^-/S_{l}^+$ is abelian, and $\wt C_{G} (S_{l+1}^-/S_{l}^+)= C_{G} ( S_{l+1}^-/S_{l}^+)$. Recall that we also have that $G/C_G(R_n)$ is unbounded. Thus by Lemma \ref{lemma:Tent--Mcpher}, there is a member $X \in \{S_j,  S_j^+, S_{j+1}^-: 0 \leq j < m_n \}$ in this series so that, setting $X^+$ to be the next member, the definable section $$H:=C_G(X) \Big / C_{C_G(X)}(X^+/X)$$ of $G$ is unbounded. Now $H$ has a definable faithful action by conjugation on the definable section $A:=X^+/X$ of $G$. Also, again by Lemma~\ref{lemma:Tent--Mcpher} (or since the infinite group $H$ acts faithfully on $A$), $A$ is infinite. Thus, there is $l \in \{0, \dots , m_{n}-1\}$ such that $A := S_{l+1}^-/S_l^+$, so we also get that $\wt C_G(A) = C_G(A)$ and that $A$ is abelian.
\end{proof}

\subsection{Almost centralizers in ultraproducts with irreducible action on abelian groups}In \cite{Tent-Macpherson2007} one of the crucial steps of the proof of the main theorem is to show that, in the context of definably hereditarily $\M_c$--groups, if $G$ is an ultraproduct $\prod_{i\in I}G_i/\mathcal{U}$ of finite groups $G_i$, where $\mathcal{U}$--many of the groups $G_i$ acts faithfully by conjugation on an abelian group $V_i$, then any definable abelian normal subgroup $A$ of $G$ is pseudo--(finite and $n$--generated). This is useful since results from finite group theory then allow one to show that, if in addition $\mathcal{U}$--many of the finite groups $G_i$ are solvable, then $G/C_G(A)$ is solvable. All this can be proven in definably hereditarily $\wt{\mathfrak{M}}_c$--context as well (Proposition~\ref{propo:ultraproduct}). Our key observation is that irreducibility of the action in the finite level allows us to show that certain almost centralizers are equal to centralizers in the infinite level. The main reason why we want to present such a tool is the hope that it can be useful in the studies of pseudo--finite permutation groups, which is a deep and recently very active topic in model theory.

We first need to introduce one more background result.

A family $\mathcal{H}$ of subgroups of $G$ is \emph{uniformly commensurable} if there is $n\in \mathbb{N}$ so that $|H_1 : H_1\cap H_2| < n$ for all $H_1,H_2 \in \mathcal{H}$. The following result is due to Schlichting but the formulation we give here can be found in \cite[Theorem 4.2.4]{Wagner2000}.

\begin{theorem}[Schlichting's Theorem]\label{th:commensurable}Let $G$ be a group and $\mathcal{H}$ be a uniformly commensurable family of subgroups of $G$. Then there is a subgroup $N$ of $G$ which is uniformly commensurable with all members of $\mathcal{H}$ and is invariant under all automorphisms of $G$ which fix $\mathcal{H}$ setwise. In fact, $N$ is a finite extension of a finite intersection of elements of $\mathcal{H}$. In particular, if $\mathcal{H}$ consists of definable subgroups, then $N$ is definable. \end{theorem}

\begin{propo}\label{propo:ultraproduct}Let $G$ be a non--principal ultraproduct of finite groups $G_i$. Suppose that $G$ acts by conjugation on a group $V$ which is a non--principal ultraproduct of finite abelian groups $V_i$. Suppose further that the groups $G$, $V$ and the action $G \curvearrowright V$ are all interpretable in some definably hereditarily \Mgp . \begin{enumerate}[(a)]
    \item Assume that $K=\prod_{i\in I}K_i/\mathcal{U}$ is a definable subgroup of $G$, and that $W=\prod_{i\in I}W_i /\mathcal{U}$ is an infinite $K$--invariant subgroup of $V$ so that, for $\mathcal{U}$--many $i$, $W_i$ is $K_i$--irreducible. Then $\wt{C}_{K}(W)=C_{K}(W)$.
    \item Suppose further that the action $G \curvearrowright V$ is faithful. If $M=\prod_{i\in I}M_i/\mathcal{U}$ is a definable normal abelian subgroup of $G$ then $M$ is pseudo--(finite and $n$--generated), that is, there is $n\in\mathbb{N}$ so that for $\mathcal{U}$--many $i$ the group $M_i$ is $n$--generated.
\end{enumerate}

\end{propo}

\begin{proof}We start by proving (a). We first find a definable $K$--invariant subgroup $D$ of $V$ such that $ W \ls D$ and $\wt{C}_{K}(D) = \wt{C}_{K}(W)$. To do so, using the $\widetilde{\mathcal{\mathfrak{M}}}_c$--condition, we find $k_1, \dots, k_l$ in $\wt{C}_{K}(W)$ and $d \in \N$, such that for $D_0 := \bigcap_{i=1}^l C_{W}(k_i)$ we have that $[D_0 : C_{D_0}(k)]< d$ for all $k \in \wt{C}_{K}(W) $. Note that $D_0$ is definable, $ W \ls D_0$ as $k_1, \dots, k_l$ in $\wt{C}_{K}(W)$ and $\wt{C}_{K}(D_0) = \wt{C}_{K}(W)$. So it is left to make it $K$--invariant. To do so, we consider the uniformly commensurable family $\{ D_0^k : k \in \wt{C}_{K}(W)\}$ of definable subgroups. As $W$ is $K$--invariant, this family is also $K$--invariant. Using Theorem~\ref{th:commensurable}, we find a $K$--invariant definable subgroup $D$ which is commensurable with all members of the family, in particular with $D_0$. This is the desired group $D$ from the beginning of the proof. Now to prove (a) it is enough to prove the following claim.

\begin{claim}
     $\wt{C}_{K}(D) \leqslant C_{K}(W)$.
\end{claim}

\begin{proof}Denote $D_1:=\wt{C}_{D}(\wt{C}_{K}(D))$. Since we are in the $\wt{\mathfrak{M}_c}$--context, we have that the $K$--invariant group $D_1$ is definable and, by Fact \ref{fact_AlCen}(\ref{item_Sym}), $D\sim D_1$, that is, $D_1$ is a finite index subgroup of $D$. So $D_1$ is equal to the ultraproduct $\prod_{i\in I}D_i^1/\mathcal{U}$, where (for $\mathcal{U}$ many $i$), $D_i^1$ is a $K_i$--invariant subgroup of $D_i$ (with the notation $D=\prod_{i\in I}D_i /\mathcal{U}$). Moreover, since $|D/D_1|=p$ for some $p\in \mathbb{N}$, $|D_i/D_i^1|=p$ for $\mathcal{U}$--many $i$. Then, by $K_i$--irreducibility, either $W_i \cap D_i^1=W_i$ or $W_i \cap D_i^1=1$. For $\mathcal{U}$--many $i$ the latter cannot happen since $W \ls D$ is infinite and $|D_i/D_i^1|=p$. This means that $W \leqslant D_1$. Consider then $[D_1, \wt{C}_{K}(D)]$. By Fact \ref{fact_AlCen}(\ref{item_GenVonNeumannThm}) this is a finite (and thus definable) subgroup of $D$ and clearly it is $K$--invariant. Since $W_i$ is $K_i$--irreducible, by repeating the argument above in the finite level, we see that $[W, \wt{C}_{K}(D)] \subset W \cap [D_1, \wt{C}_{K}(D)]=1$. That is $\wt{C}_{K}(D) \leqslant C_{K}(W)$. \end{proof}

We move on to prove part (b). Suppose contrary that (for $\mathcal{U}$--many $i$) the abelian group $M_i$ is not $n$--generated. In what follows, we construct a sequence of subgroups $W_i^1, W_i^2, \ldots ,W_i^n$ of $V_i$ in such way that, setting $W_j=\prod_{i\in I}W_{i}^j/\mathcal{U}$, for any $j\in \{1, \ldots, n\}$, we have \begin{enumerate}
    \item $\wt C_M(W_1+\cdots + W_j)=C_M(W_1+\cdots + W_j)$, and
    \item $C_{M_i}(W_i^1)> C_{M_i}(W_i^1+ W_i^2) > \ldots > C_{M_i}(W_i^1+W_i^2+\cdots + W_i^n)$.
\end{enumerate} 
Since $n$ is arbitrary, this construction gives us an infinite chain of almost centralizers in $M$, thus provides a contradiction.

Now we argue as in \cite[Proof of Claim 3]{Tent-Macpherson2007}. First choose a minimal $G_i$--invariant normal subgroup $A_i$ of $V_i$. Then $A_i$ is a vector space over some prime field, and $G_i$ acts linearly on it. Now choose $W_i^1$ to be any non--trivial subspace of $A_i$ which is $M_i$--irreducible. Then there is an element $w_i \in W_i^1$ so that $C_{M_i}(W_i^1)=C_{M_i}(w_i^{M_i})$. Set $M_i^1=C_{M_i}(W_i^1)$. By Schur's Lemma, $M_i/M_i^1$ is cyclic, so $M_i^1$ is not $(n-1)$--generated. If we have found $W_i^1,\ldots, W_i^r$ and $M_i^1, \ldots, M_i^r$, choose $W_i^{r+1}$ to be any non--trivial $M_i^{r}$--irreducible subspace of $A_i$, and $M_i^{r+1}=M_i^{r}/C_{M_i^{r}}(W_i^{r+1})$. Then, by induction, $M_i^{r+1}$ is not $(n-r-1)$--generated. So we get the sequence as in (2). Since we are in the definably hereditarily $\wt{\mathfrak{M}_c}$--context, it directly follows from part (a) that (1) holds as well. So (b) is proven.
    \end{proof}

\addcontentsline{toc}{section}{References}    
\bibliographystyle{plain}
\bibliography{radical}

\end{document}